\documentclass[a4paper,english,fontsize=11pt,parskip=half,abstract=true]{scrartcl}
\usepackage{babel}
\usepackage[utf8]{inputenc}
\usepackage[T1]{fontenc}
\usepackage[a4paper,left=20mm,right=20mm,top=30mm,bottom=30mm]{geometry}
\usepackage{amsmath}
\usepackage{amsthm}
\usepackage{amssymb}
\usepackage{enumerate}
\usepackage{thmtools}
\usepackage[bookmarks=true,
            pdftitle={Broué's Conjecture for 2-blocks with elementary abelian defect groups of order 32},
            pdfauthor={Benjamin Sambale},
            pdfkeywords={2-blocks, Broué's Conjecture, abelian defect groups},
            pdfstartview={FitH}]{hyperref}

\newtheorem{Thm}{Theorem} %[section]

\numberwithin{equation}{section}

\allowdisplaybreaks[1]

\renewcommand{\phi}{\varphi}
\newcommand{\C}{\mathrm{C}}
\newcommand{\N}{\mathrm{N}}
\newcommand{\Z}{\mathrm{Z}}

\newcommand{\Aut}{\operatorname{Aut}}

\newcommand{\SL}{\operatorname{SL}}

\newcommand{\cO}{\mathcal{O}}

\mathchardef\ordinarycolon\mathcode`\:  %defines a nice ":=" 
 \mathcode`\:=\string"8000
 \begingroup \catcode`\:=\active
   \gdef:{\mathrel{\mathop\ordinarycolon}}
 \endgroup

\title{Broué's Conjecture for 2-blocks with elementary abelian defect groups of order 32}
\author{Cesare Giulio Ardito\footnote{Department of Mathematics, City University of London, Northampton Square, London, EC1V 0HB, UK, \href{mailto:cesare.ardito@city.ac.uk}{cesare.ardito@city.ac.uk}} \ and Benjamin Sambale\footnote{Institut für Algebra, Zahlentheorie und Diskrete Mathematik, Leibniz Universität Hannover, Welfengarten 1, 30167 Hannover, Germany,
\href{mailto:sambale@math.uni-hannover.de}{sambale@math.uni-hannover.de}}}
\date{\today}

\begin{document}
\frenchspacing
\maketitle
\begin{abstract}\noindent
The first author has recently classified the Morita equivalence classes of $2$-blocks $B$ of finite groups with elementary abelian defect group of order $32$. In all but three cases he proved that the Morita equivalence class determines the inertial quotient of $B$. We finish the remaining cases by utilizing the theory of lower defect groups. As a corollary, we verify Broué's Abelian Defect Group Conjecture in this situation.
\end{abstract}


Motivated by Donovan's Conjecture in modular representation theory, there has been some interest in determining the possible Morita equivalence classes of $p$-blocks $B$ of finite groups over a complete discrete valuation ring $\cO$ with a given defect group $D$. While progress in the case $p>2$ seems out of reach at the moment, quite a few papers appeared recently addressing the situation where $D$ is an abelian $2$-group. For instance, in \cite{EatonE8,EatonE16,EKKS,EL2,WuZhang} a full classification was obtained whenever $D$ is an abelian $2$-group of rank at most $3$ or $D\cong C_2^4$. Building on that, the first author determined in \cite{Ardito} the Morita equivalence classes of blocks with defect group $D\cong C_2^5$. Partial results on larger defect groups were given in \cite{Ardito64,ArditoMcKernon,McKernon}.

Since every Morita equivalence is also a derived equivalence, it is reasonable to expect that Broué's Abelian Defect Group  Conjecture for $B$ follows once all Morita equivalences have been identified. It is however not known in general whether a Morita equivalence preserves inertial quotients. In fact, there are three cases in \cite[Theorem~1.1]{Ardito} where the identification of the inertial quotient was left open. We settle these cases by making use of lower defect groups. Our notation follows \cite{habil}. All blocks are considered over $\cO$. 

\begin{Thm}\label{main}
Let $B$ be a $2$-block of a finite group $G$ with defect group $D\cong C_2^5$. Then the Morita equivalence class of $B$ determines the inertial quotient of $B$.
\end{Thm}
\begin{proof}
By \cite[Theorem~1.1]{Ardito}, we may assume that $B$ is Morita equivalent to the principal block of one of the following groups:
\begin{enumerate}[(i)]
\item $(C_2^4\rtimes C_5)\times C_2$.
\item\label{caseb} $(C_2^4\rtimes C_{15})\times C_2$.
\item\label{casec} $\SL(2,16)\times C_2$.
\end{enumerate}
Assume the first case. The elementary divisors of the Cartan matrix $C$ of $B$ (a Morita invariant) are $2,2,2,2,32$. According to \cite[Corollary~5.3]{Ardito}, we may assume by way of contradiction that $B$ has inertial quotient $E\cong C_7\rtimes C_3$ such that $\C_D(E)=1$.
There is an $E$-invariant decomposition $D=D_1\times D_2$ where $|D_1|=4$.
Let $(Q,b)$ be a $B$-subpair such that $|Q|=2$ (i.\,e. $b$ is a Brauer correspondent of $B$ in $\C_G(Q)$). Then $b$ dominates a unique block $\overline{b}$ of $\C_G(Q)/Q$ with defect $4$. The possible Cartan matrices of such blocks have been computed in \cite[Proposition~16]{SambaleC4} up to basic sets.
If $Q\le D_1$, then $b$ has inertial quotient $\C_E(Q)\cong C_7$ (see \cite[Lemma~1.34]{habil}) and the Cartan matrix $C_b$ of $b$ has elementary divisors $4,4,4,4,4,4,32$. By \cite[Eq. (1.2) on p. 16]{habil}, the $1$-multiplicity $m^{(1)}_b(Q)$ of $Q$ as a lower defect group of $b$ is $0$. But now also $m_B^{(1)}(Q,b)=0$ by \cite[Lemma~1.42]{habil}. Similarly, if $Q\nsubseteq D_1\cup D_2$, then $b$ is nilpotent and again $m_B^{(1)}(Q,b)=0$. Finally let $Q\le D_2$. Then $b$ has inertial index $3$ and $C_b$ has elementary divisors $2,2,32$. In particular, $m^{(1)}_B(Q,b)= m^{(1)}_b(Q)\le 2$. Since all subgroups of order $2$ in $D_2$ are conjugate under $E$, the multiplicity of $2$ as an elementary divisor of $C$ is at most $2$ by \cite[Proposition~1.41]{habil}. Contradiction.

Now assume that case \eqref{caseb} or \eqref{casec} occurs. In both cases the multiplicity of $2$ as an elementary divisor of $C$ is $14$. By \cite[Corollary~5.3]{Ardito}, we may assume that $E\cong (C_7\rtimes C_3)\times C_3$. Again we have an $E$-invariant decomposition $D=D_1\times D_2$ where $|D_1|=4$. As above let $Q\le D$ with $|Q|=2$. If $Q\le D_1$, then $b$ has inertial quotient $C_7\rtimes C_3$ and the elementary divisors of $C_b$ are all divisible by $4$. Hence, $m_B^{(1)}(Q,b)=0$. If $Q\nsubseteq D_1\cup D_2$, then $b$ has inertial index $3$ and $C_b$ has elementary divisors $8,8,32$. Again, $m_B^{(1)}(B,b)=0$. Now if $Q\le D_2$, then $b$ has inertial quotient $C_3\times C_3$. Here either $l(b)=1$ or $C_b$ has elementary divisors $2,2,2,2,8,8,8,8,32$. As above we obtain $m_B^{(1)}(Q,b)\le 4$. Thus, the multiplicity of $2$ as an elementary divisor of $C$ is at most $4$. Contradiction. 
\end{proof}

Now we are in a position to prove Broué's Conjecture in the situation of \autoref{main}.

\begin{Thm}\label{broue}
Let $B$ be a $2$-block of a finite group $G$ with defect group $D\cong C_2^5$. Then $B$ is derived equivalent to its Brauer correspondent $b$ in $\N_G(D)$.
\end{Thm}
\begin{proof}
Let $E$ be the inertial quotient of $B$ (and of $b$). We first prove Alperin's Weight Conjecture for $B$, i.\,e. $l(B)=l(b)$. 
By \cite[Corollary~5.3]{Ardito}, $E$ uniquely determines $l(B)$ (and $l(b)$) unless $E\in\{C_3^2, (C_7\rtimes C_3)\times C_3\}$. Suppose first that $E=C_3^2$. Then $\C_D(E)=\langle x\rangle\cong C_2$. Let $\beta$ be a Brauer correspondent of $B$ in $\C_G(D)$ such that $b=\beta^{N}$ where $N:=\N_G(D)$. A theorem of Watanabe~\cite{Watanabe1} (see \cite[Theorem~1.39]{habil}) shows that $l(B)=l(B_x)$ where $B_x:=\beta^{\C_G(x)}$. As usual $B_x$ dominates a block $\overline{B_x}$ of $\C_G(x)/\langle x\rangle$ with defect $4$ such that $l(B_x)=l(\overline{B_x})$. Since Alperin's Conjecture holds for $2$-blocks of defect $4$ (see \cite[Theorem~13.6]{habil}), we obtain $l(\overline{B_x})=l(\overline{b_x})$ where $\overline{b_x}$ is the unique block of $\C_N(x)/\langle x\rangle$ dominated by $b_x:=\beta^{\C_N(x)}$. Hence, 
\[l(B)=l(B_x)=l(\overline{B_x})=l(\overline{b_x})=l(b_x)=l(b)\] 
as desired. 
Next, we assume that $E=(C_7\rtimes C_3)\times C_3$. Up to $G$-conjugacy there exist three non-trivial $B$-subsections $(x,B_x)$, $(y,B_y)$ and $(xy,B_{xy})$. The inertial quotients are $E(B_x)=C_3^2$, $E(B_y)=C_7\rtimes C_3$ and $E(B_{xy})=C_3$. By \cite[Corollary~5.3]{Ardito}, $l(B_y)=5$, $l(B_{xy})=3$ and $(k(B),l(B))\in\{(32,15),(16,7)\}$. Since $k(B)-l(B)=l(B_x)+l(B_y)+l(B_{xy})$, we obtain as above
\[l(B)=15\Longleftrightarrow l(B_x)=9\Longleftrightarrow l(b_x)=9\Longleftrightarrow l(b)=15.\]
This proves Alperin's Conjecture for $B$. 

Now suppose that the Morita equivalence class of $B$ is given as in \cite[Theorem~1.1]{Ardito}. Then $k(B)$ can be computed and $E$ is uniquely determined by \autoref{main}. By \cite[Corollary~5.3]{Ardito}, also the action of $E$ on $D$ is uniquely determined. 
By a theorem of Külshammer~\cite{Kuelshammer} (see \cite[Theorem~1.19]{habil}), $b$ is Morita equivalent to a twisted group algebra of $D\rtimes E$. The corresponding $2$-cocycle is determined by $l(b)=l(B)$ (see \cite[proof of Theorem~5.1]{Ardito}). Hence, we have identified the Morita equivalence class of $b$ and it suffices to check Broué's Conjecture for the blocks listed in \cite[Theorem~1.1]{Ardito}. 

For the solvable groups in that list, we have $G=N$ and $B=b$. For principal $2$-blocks, Broué's Conjecture has been shown in general by Craven and Rouquier~\cite[Theorem~4.36]{CravenRouquier}. 
Now the only remaining case in \cite[Theorem~1.1]{Ardito} is a non-principal block $B$ of 
\[G:=(\SL(2,8)\times C_2^2)\rtimes 3^{1+2}_+.\] 

As noted in \cite[Remark 3.4]{Okuyama}, the splendid derived equivalence between the principal block of $\SL(2,8)$ and its Brauer correspondent extends to a splendid derived equivalence between the principal block of $\Aut(\SL(2,8))$ and its Brauer correspondent. An explicit proof of this fact can be found in \cite[Section~6.2.1]{CravenRouquier}.
Let $M\cong \SL(2,8) \times C_3 \times A_4$ be a normal subgroup of $G$ such that $C_3\cong \Z(G) \leq M$, and let $B_M$ be the unique block of $M$ covered by $B$. By composing the derived equivalence from \cite{Okuyama} with a trivial Morita equivalence, we deduce that $B_M$ is splendid derived equivalent to its Brauer correspondent.
Using the notation of \cite[Theorem 3.4]{Marcus}, the complex that defines this equivalence extends to a complex of $\Delta$-modules, which follows from the remark above and the fact that the trivial Morita equivalence naturally extends (noting that $G/M$ stabilizes each block of $M$). Therefore, by \cite[Theorem 3.4]{Marcus}, $B$ is derived equivalent to $b$.
\end{proof}

Note that we do not prove that the derived equivalences in \autoref{broue} are splendid.

In an upcoming paper by Charles Eaton and Michael Livesey the $2$-blocks with abelian defect groups of rank at most $4$ are classified. It should then be possible to prove Broué's Conjecture for all abelian defect $2$-groups of order at most $32$. Judging from \cite{EL2} we expect that all blocks with defect group $C_4\times C_2^3$ are Morita equivalent to principal blocks. 

\section*{Acknowledgment}
We thank Michael Livesey for a very helpful discussion.
The first author is supported by the London Mathematical Society (\mbox{ECF-1920-03}). The second author is supported by the German Research Foundation (\mbox{SA 2864/1-2} and \mbox{SA 2864/3-1}).

\end{document}